\documentclass[10pt,a4paper]{article}
\usepackage[left=2.5cm,right=2.5cm,top=1.5cm,bottom=1.5cm]{geometry}
\usepackage[utf8]{inputenc}
\usepackage[english]{babel}
\usepackage[T1]{fontenc}
\usepackage{amsmath}
\usepackage{amsthm}
\usepackage{amsfonts}
\usepackage{amssymb}
\usepackage{dsfont}
\usepackage{graphicx}
\usepackage{kpfonts}
\usepackage[colorlinks,citecolor=blue,linkcolor=blue]{hyperref}
\usepackage{enumitem}
\usepackage{multicol}
\usepackage[lofdepth,lotdepth]{subfig}

\newcommand{\N}{\mathbb{N}}
\newcommand{\Z}{\mathbb{Z}}
\newcommand{\Q}{\mathbb{Q}}

\newcommand{\R}{\mathbb{R}}

\newcommand{\PR}{\mathbb{P}}
\newcommand{\ESP}{\mathbb{E}}

\newcommand{\Ma}[1]{\mathcal{M}_{\alpha} \left( { #1} \right)}

\def\A{{\cal A}}
\def\ce{{\cal C}}

\def\I{{\cal I}}

\def\Fa{\mathcal{F}_{\alpha}}

\def\o{\omega}
\def\O{\Omega}

\def\al{\alpha}
\def\ga{\gamma}

\def\ind{\mathds{1}}
\def\undal{\underline{\alpha}}
\def\oval{\overline{\alpha}}
\def\vep{\varepsilon}

\numberwithin{equation}{section}

\newtheorem{thm}{Theorem}[section]
\newtheorem{lemma}{Lemma}[section]

\newtheorem{definition}{Definition}[section]
\newtheorem{rem}{\textbf{Remark}}[section]

\title{Wavelet series representation for multifractional multistable Riemann-Liouville process}
\author{Antoine Ayache \\
Univ. Lille, CNRS, UMR 8524 - Laboratoire Paul Painlevé,\\ F-59000 Lille, France.\\
E-mail: \texttt{Antoine.Ayache@univ-lille.fr}\\
\and
Julien Hamonier \\
Univ. Lille, CHU Lille, ULR 2694 - METRICS : \\
Évaluation des technologies de santé et des
pratiques médicales, \\
F-59000 Lille, France.
\\
E-mail: \texttt{Julien.Hamonier@univ-lille.fr}
 }
\date{\today}

\begin{document}

\maketitle

\begin{abstract}
The main goal of this paper is to construct a wavelet-type random series representation for a random field $X$, defined by a multistable stochastic integral, which generates a multifractional multistable Riemann-Liouville (mmRL) process $Y$. Such a representation provides, among other things, an efficient method of simulation of paths of $Y$. In order to obtain it, we expand in the Haar basis the integrand associated with $X$ and we use some fundamental properties of multistable stochastic integrals. Then, thanks to the Abel's summation rule and the Doob's maximal inequality for discrete submartingales, we show that this wavelet-type random series representation of $X$ is convergent in a strong sense: almost surely in some spaces of continuous functions. Also, we determine an estimate of its almost sure rate of convergence in these spaces.
\end{abstract}

\noindent {\bf MSC2010:}  60G17, 60G22, 60G52\\
\noindent{\bf Keywords:} Stable and multistable distributions, fractional and multifractional processes, varying Hurst parameter, Haar basis, random series of functions.

\section{Introduction}
\label{sec:intro}

The main idea behind multifractional processes is that Hurst parameter which governs path roughness is no longer a constant but a function whose values can change from point to point (see e.g. \cite{Aya18}). Thus, such processes provide more flexible models than the classical fractional Brownian motion whose path roughness remains everywhere the same. In the same spirit, the articles \cite{FL09,FGL09,falconer2012multistable} have proposed three different (non-equivalent) approaches allowing to generalize stable stochastic processes (see for instance \cite{samorodnitsky:taqqu:1994book}) in such a way that the parameter $\alpha$ governing the heavy tail behaviour of their distributions becomes a function. Such generalizations are called multistable processes. 
The approach introduced in \cite{falconer2012multistable} relies on the construction of multistable stochastic integrals. Such an integral $\I$ depends on a functional parameter $\al(\cdot)$; this deterministic Lebesgue measurable function $\al (\cdot)$ is defined on the real line $\R$ and with values in some compact interval $[\undal\,, \oval ]$ included in $(0,2]$.
Throughout this article, we assume that $\al (\cdot)$ belongs to the H\"older space $\ce^{1+\rho_{\al}} ([0,1])$, for some $\rho_{\al}\in (0,1)$; in other words $\al (\cdot)$ is continuously differentiable on $[0,1]$ and its derivative $\al'(\cdot)$ satisfies a uniform H\"older condition on $[0,1]$ of order $\rho_{\al}$, that is one has $|\al' (s_1)-\al' (s_2)|\le c |s_1-s_2|^{\rho_{\al}}$, for some constant $c>0$ and for all $(s_1,s_2)\in [0,1]^2$. Moreover, we assume that 
\begin{equation}
\label{eq:condal}
1<\undal \le \al (s)\le \oval <2, \quad\mbox{for all $s\in\R$.}
\end{equation}
The integrands associated with the multistable stochastic integral $\I$ are the deterministic functions from $\R$ to $\R$ belonging to 
$\Fa$, the Lebesgue space of variable order defined as:
\begin{equation}
\label{def:fa}
 \Fa:=\Big\{\mbox{$f$  s.t. $f$ is a Lebesgue measurable function from $\R$ to $\R$ and } \int_\R |f(s)|^{\al (s)} ds < +\infty \Big\}.
\end{equation}
Notice that, for any fixed $f\in \Fa^{\ast}:=\Fa \setminus \{0\}$, the function from $(0,+\infty)$ into itself $\lambda\mapsto \int_{\R}\lambda^{-\al (s)} |f(s)|^{\al (s)} ds$ is continuous and strictly decreasing, and one has 
$\lim_{\lambda\rightarrow 0^+}\int_{\R}\lambda^{-\al (s)} |f(s)|^{\al (s)} ds=+\infty$ and $\lim_{\lambda\rightarrow +\infty}\int_{\R}\lambda^{-\al (s)} |f(s)|^{\al (s)} ds=0$. Therefore, there exists a unique positive real number denoted by $\|f\|_{\alpha}$ such that 
\begin{equation}
\label{eq:def-normF}
 \int_{\R}\|f\|_{\alpha}^{-\al (s)} |f(s)|^{\al (s)} ds=1.
 \end{equation}
The map $\|\cdot\|_{\alpha}$ defined on $\Fa$ in this way and by using the convention that $\| 0\|_{\al}=0$ is a quasi-norm on $\Fa$; notice that the only difference between a norm and a quasi-norm is that in the latter case the triangular inequality holds up to a multiplicative constant, namely there exists $c'\in [1,+\infty)$, such that $\|f+g\|_\al \le c'\big (\| f\|_\al+ \|g \|_\al \big)$, for all $f,g\in\Fa$. Also notice that one can derive from (\ref{def:fa}) and the inequality 
\[
 \int_\R |f(s)|^{\al (s)} ds\le  \int_\R |f(s)|^{\undal}\,ds+ \int_\R |f(s)|^{\oval} ds,
 \]
which is satisfied by any Lebesgue measurable function $f$ from $\R$ to $\R$, that
\begin{equation}
\label{eq:faint}
L^{\undal}\,(\R)\cap L^{\oval}(\R)\subseteq \Fa \,,
\end{equation}
where, for all $p\in (0,+\infty]$, one denotes $L^p (\R)$ the classical Lebesgue space of order $p$ of real-valued functions over $\R$. Moreover, there is a finite constant $\kappa_1$ only depending on $\undal$ and $\oval$, such that, for all $f\in L^{\undal}\,(\R)\cap L^{\oval}(\R)$, one has
\begin{equation}
\label{eq:inegFa}
\| f\|_{\al}\le \kappa_1 \big (\| f\|_{\undal}+\| f\|_{\oval}\big)=\kappa_1\Big (\int_\R |f(s)|^{\undal}\,ds\Big)^{1/\,\undal}+ \kappa_1\Big (\int_\R |f(s)|^{\,\oval} ds \Big)^{\,1/\oval}.
\end{equation}
The latter inequality simply results from the fact that 
\begin{eqnarray*}
&& \int_\R \Big (2^{1/\,\undal} \| f\|_{\undal}+2^{1/\,\oval} \| f\|_{\oval}\Big)^{-\al (s)} |f(s)|^{\al (s)} ds\\
&& \le \int_\R \Big (2^{1/\,\undal} \| f\|_{\undal}+2^{1/\,\oval} \| f\|_{\oval}\Big)^{-\undal} \, |f(s)|^{\undal} \, ds
+\int_\R \Big (2^{1/\,\undal} \| f\|_{\undal}+2^{1/\,\oval} \| f\|_{\oval}\Big)^{-\oval} |f(s)|^{\oval} ds\\
&& \le 2^{-1}\int_\R  \| f\|_{\undal}^{-\undal} \, |f(s)|^{\undal} \, ds+2^{-1}\int_\R  \| f\|_{\oval}^{-\oval} \, |f(s)|^{\oval} \, ds=1.
\end{eqnarray*}

Let us now recall some fundamental properties of the multistable stochastic integral $\I$ which was introduced in \cite{falconer2012multistable}. Denote by $L^{\gamma} (\O,\A, \PR)$ the space of the real-valued random variables on a given probability space $(\O,\A, \PR)$ whose absolute moment of order $\gamma$ is finite, where $\gamma\in (0,\undal)$ is arbitrary and fixed. The integral $\I$ is a linear map from $\Fa$ into $L^{\gamma} (\O,\A, \PR)$ such that, for all $f\in \Fa$, the characteristic function $\Phi_{\I (f)}$ of the random variable $\I (f)$ satisfies 
\begin{equation}
\label{eq:charI}
\Phi_{\I (f)}(\xi):=\ESP \big ( e^{i \xi \I (f)}\big)=\exp\Big (-\int_{\R} \big | \xi f(s)\big |^{\al (s)} ds\Big)\,,\quad\mbox{for every $\xi\in\R$.}
\end{equation}
Observe that (\ref{eq:charI}) implies that $\I (f)$ has a symmetric distribution. Similarly to stable stochastic integrals (see for instance \cite{samorodnitsky:taqqu:1994book}), one can in a natural way associate to the multistable stochastic integral $\I$ an independently scattered random measure denoted by 
$\mathcal{M}_\al$ (see \cite{falconer2012multistable}). Thus, $\I (f)$ is frequently denoted by $\int_\R f(s) \Ma{ds}$. It is worth mentioning that an upper bound of the asymptotic behavior at $+\infty$ of the tail of the distribution of the random variable $\int_\R f(s) \Ma{ds}$ is provided by Proposition~2.3 of \cite{falconer2012multistable}: 
\begin{equation}
\label{eq:taildist1}
\PR\bigg (\Big | \int_\R f(s) \Ma{ds}\Big |\ge \lambda \bigg)\le \kappa_2 \int_{\R}\lambda^{-\al (s)} |f(s)|^{\al (s)} ds,\quad\mbox{for all $\lambda\in (0,+\infty)$,} 
\end{equation}
where $\kappa_2$ is a constant only depending on $\undal$ and $\oval$. The same proposition also provides, thanks to (\ref{eq:taildist1}), an estimate for the absolute moment of any order $\gamma\in (0,\undal)$ of $\int_\R f(s) \Ma{ds}$:
\begin{equation}
\label{eq:taildist2}
\ESP \bigg (\Big | \int_\R f(s) \Ma{ds}\Big |^\gamma \bigg)\le \kappa_3 (\gamma)\,\| f\|_\al^{\gamma},
\quad\mbox{for each fixed $\gamma\in (0,\undal)$,}
\end{equation}
where $\kappa_3 (\gamma)$ is a finite constant only depending on $\gamma$, $\undal$ and $\oval$. We mention in passing that the paper \cite{Aya13} has shown that the inequality (\ref{eq:taildist1}) is sharp: the reverse inequality also holds.


\section{Main result and simulations}
\label{sec:stat}

Let us now give the main motivation behind our present article. The paper \cite{ham2015} has introduced via Haar basis an almost surely uniformly convergent wavelet-type random series representation for the stable stochastic field which generates linear multifractional stable motions \cite{stoev2004stochastic,stoev2005path}. In our present article, we intend to generalize this result to the framework of the multistable stochastic field which generates linear multifractional multistable motions of Riemann-Liouville type. The latter field is denoted by $\big\{X(u,v)\,:\, (u,v)\in [0,1]\times (1/\undal\,,1)\big\}$, and defined, for all $(u,v)\in [0,1]\times (1/\undal\,,1)$, as:
\begin{equation}
\label{def:Xgenerated}
X(u,v):=\int_\R K_{u,v}(s) \Ma{ds}, 
\end{equation}
where, for every $(u,v,s)\in [0,1]\times (1/\undal\,,1)\times \R$,
\begin{equation}\label{def_Kuv}
K_{u,v}(s):= (u-s)_+^{v-\frac{1}{\alpha(s)}}\ind_{[0,1]}(s)=
\left\lbrace
\begin{array}{ll}
0 & \mbox{ if }s\notin [0,u),\\
(u-s)^{v-\frac{1}{\alpha(s)}} & \mbox{ otherwise.}
\end{array}
\right.
\end{equation}
It can easily be seen that, for each fixed $(u,v)\in [0,1]\times (1/\undal\,,1)$, one has 
\begin{equation}
\label{eq:ineg-kuv}
0\le K_{u,v}(s) \le \ind_{[0,1]}(s), \quad\mbox{for all $s\in\R$.}
\end{equation}
Thus, the function $K_{u,v}$ belongs to the space $\Fa$ (see (\ref{def:fa})) which guarantees the existence of the multistable stochastic integral in (\ref{def:Xgenerated}). Also, one can derive from (\ref{eq:ineg-kuv}) that the function $K_{u,v}$ belongs to all the Lebesgue spaces $L^p ([0,1])$, $p\in (0,+\infty]$, and in particular to the Hilbert space $L^2 ([0,1])$. A well-known orthonormal basis of the latter space was introduced by Haar in \cite{haar}; it consists in the following collection of functions: 
\begin{equation}
\label{A:eq:haar}
\left\{
\begin{array}{l}
\ind_{[0,1)}(\bullet)\, ,\\
 2^{j/2} h(2^j\bullet-k) = 2^{j/2} \left(\ind_{\big[ 2^{-j}k, 2^{-j}(k+\frac{1}{2})\big)}(\bullet) - \ind_{\big[ 2^{-j}(k+\frac{1}{2}), 2^{-j}(k+1)\big)}(\bullet) \right), \quad  j \in \Z_+\mbox{ and } k \in \{ 0, \dots, 2^j-1\},
\end{array}
\right.
\end{equation}
where $h:=\ind_{[0,1/2)}-\ind_{[1/2,1)}$. By expanding, for each fixed $(u,v)\in [0,1]\times (1/\undal\,,1)$, the function $K_{u,v}$ on the latter basis, one gets that
\begin{equation}
\label{kuv:decompLoval}
K_{u,v}(\bullet)=  \| K_{u,v}\|_1\ind_{[0,1)}+\sum_{j=0}^{+\infty} \sum_{k=0}^{2^j-1} w_{j,k} (u,v) h(2^j\bullet-k)\,, 
\end{equation}
where $\| K_{u,v}\|_1:=\int_0^1 K_{u,v} (s)ds$ and 
\begin{equation}
\label{eq:def-wjk}
w_{j,k} (u,v):= 2^j \int_0 ^1 K_{u,v}(s) h(2^js-k) ds, \quad\mbox{for all $j\in\Z_+$ and $k\in\{0,\ldots, 2^j-1\}$.} 
\end{equation}
A priori, the series in (\ref{kuv:decompLoval}) is convergent for the $L^2([0,1])$ norm; yet (\ref{eq:condal}), (\ref{eq:inegFa}) and the H\"older inequality imply that this series is also convergent for the quasi-norm $\|\cdot\|_\al$. Thus,
using (\ref{def:Xgenerated}) and (\ref{eq:taildist2}) one gets that
\begin{equation}
\label{eq:ser-X}
X(u,v)=\| K_{u,v}\|_1\eta +\sum_{j=0}^{+\infty} \sum_{k=0}^{2^j-1} w_{j,k} (u,v)\vep_{j,k}\,,
\end{equation}
where $\eta:=\int_\R \ind_{[0,1)}(s) \Ma{ds}=\mathcal{M}_\al \big ([0,1)\big)$ and 
\begin{equation}
\label{eq:def-epjk}
\vep_{j,k}:=\int_\R h(2^j s-k) \Ma{ds}, \quad\mbox{for all $j\in\Z_+$ and $k\in\{0,\ldots, 2^j-1\}$.} 
\end{equation}
A priori, the series in (\ref{eq:ser-X}) is convergent in the sense of the $L^{\gamma} (\O,\A, \PR)$ (quasi)-norm, for each fixed $(u,v)\in [0,1]\times (1/\undal\,,1)$ and $\gamma\in (0,\undal)$. The main goal of our article is to show that it is also convergent in a much stronger sense, namely:
\begin{thm}
\label{theo:main}
For all integer $J\ge 1$ and $(u,v)\in [0,1]\times (1/\undal,1)$, let $X^J (u,v)$ be the partial sum of the series in (\ref{eq:ser-X}) defined as:
\begin{equation}
\label{theo:main:eq2}
X^J (u,v)=\| K_{u,v}\|_1\eta +\sum_{j=0}^{J-1} \sum_{k=0}^{2^j-1} w_{j,k} (u,v)\vep_{j,k}\,.
\end{equation}
Then, there exists an event $\O^*$ of probability~$1$, such that, for all $\o\in\O^*$ and for every real numbers $a$ and $b$ satisfying $1/\undal<a<b<1$, $\big (X^J (\cdot,\cdot, \o)\big)_{J\in\N}$ is a Cauchy sequence in $\ce\big ([0,1]\times [a,b]\big)$ the Banach space of the real-valued continuous functions over the rectangle $[0,1]\times [a,b]$ equipped with the uniform norm denoted by $\|\cdot\|_{\ce}$. Thus, it is convergent in this space. Moreover, the multistable stochastic field $\big\{\widetilde{X}(u,v)\,:\, (u,v)\in [0,1]\times [a,b]\big\}$  with continuous paths, defined as:
\begin{equation}
\label{theo:main:eq3}
\widetilde{X}(\cdot,\cdot,\o):=\lim_{J\rightarrow +\infty} X^J (\cdot,\cdot, \o),\,\mbox{if $\o\in\O^*$,}\quad\mbox{and}\quad\widetilde{X}(\cdot,\cdot,\o):=0,\, \mbox{else,}
\end{equation}
is a modification of $\big\{X(u,v)\,:\, (u,v)\in [0,1]\times [a,b]\big\}$, and one has, for any fixed $\zeta>1/\undal$ and $\o\in\O^*$,
\begin{equation}
\label{theo:main:eq4}
\sup\Big\{J^{-\zeta}\, 2^{J\min\{\rho_{\al},\,a-1/\undal\,\}}\big | \widetilde{X}(u,v,\o)-X^J (u,v,\o)\big| \,:\, (J,u,v)\in \N\times [0,1]\times [a,b]\Big\}<+\infty.
\end{equation}
\end{thm}

\begin{rem}
\label{improv:stable}
In view of (\ref{theo:main:eq2}) and of the fact that $\widetilde{X}$ is a modification of $X$, the inequality (\ref{theo:main:eq4}) provides an almost sure estimate of the rate of convergence for the uniform norm $\|\cdot\|_{\ce}$ of the random series of functions in (\ref{eq:ser-X}). Notice that in the particular case where $X$ is an $\al$-stable field (that is $\al(s)=\al$, for all $s\in [0,1]$, where $\al \in (1,2)$ is a constant parameter), this estimate of the rate of convergence becomes $\mathcal{O}(2^{-J(a-1/\al)} J^{1/\al+\eta})$, where $\eta$ is an arbitrarily small fixed positive real number. Thus, it improves the estimate $\mathcal{O}(2^{-J(a-1/\al)} J^{2/\al+\eta})$ which was previously obtained in \cite[Theorem~2.1]{ham2015}. 
\end{rem}

\begin{definition}
\label{def:lmmm}
Let $H(\cdot)$ be a deterministic function from $[0,1]$ into $[a,b]\subset (1/\undal,1)$. The  multifractional multistable Riemann-Liouville (mmRL) process of parameter $H(\cdot)$, generated by the field $\big\{\widetilde{X}(u,v)\,:\, (u,v)~\in~[0,1]~\times~ [a,b]\big\}$, is the multistable process denoted by $\{Y(t)\,:\, t\in [0,1]\}$ and defined as:
\begin{equation}
\label{def:lmmm:eq1}
Y(t):=\widetilde{X}(t,H(t)), \quad\mbox{for all $t\in [0,1]$.}
\end{equation}
Notice that when the function $H(\cdot)$ is a constant $\{Y(t)\,:\, t\in [0,1]\}$ is called fractional multistable Riemann-Liouville (fmRL) process.
\end{definition}

\begin{rem}
\label{rem1:lmmm}
It easily follows from Theorem~\ref{theo:main} and Definition~\ref{def:lmmm} that $\{Y(t)\,:\, t\in [0,1]\}$ has almost surely continuous paths as soon as $H(\cdot)$ is a continuous function.
\end{rem}

\begin{rem}
\label{rem1:fieldX}
Using (\ref{A:eq:haar}), (\ref{eq:def-wjk}), and (\ref{eq:def-epjk}), it can be shown by induction on $J$ that, for all $J\in\N$ and for each $(u,v)\in [0,1]\times (1/\undal,1)$, the random variable $X^J (u,v)$, defined in (\ref{theo:main:eq2}), can be expressed as:
\begin{equation}
\label{rem1:fieldX:eq1}
X^J (u,v)=\sum_{l=0}^{2^J-1} \overline{K}_{u,v}^{J,l} \, \mathcal{M}_\al \Big (\big [2^{-J} l, 2^{-J}(l+1)\big)\Big)\, ,
\end{equation}
where, for all $J\in\N$ and $l\in\{0,\ldots, 2^J-1\}$, $\overline{K}_{u,v}^{J,l}$ is the average value of the function $K_{u,v}$ on the dyadic interval $\big [2^{-J} l, 2^{-J}(l+1)\big)$, that is 
\begin{equation}
\label{rem1:fieldX:eq2}
\overline{K}_{u,v}^{J,l}:=2^{J}\int_{2^{-J} l}^{2^{-J}(l+1)} K_{u,v}(s)ds.
\end{equation}
The equalities (\ref{rem1:fieldX:eq1}), (\ref{theo:main:eq3}) and (\ref{def:lmmm:eq1}) provide an efficient method for simulating paths of the mmRL process $Y$. To this end, when $J$ is large enough, one uses the approximation:
\begin{equation}
\label{rem1:fieldX:eq3}
\mathcal{M}_{\alpha} \Big (\big [2^{-J} l, 2^{-J}(l+1)\big)\Big) \approx \mathcal{Z}_{\alpha(2^{-J}l)}  \Big (\big [2^{-J} l, 2^{-J}(l+1)\big)\Big)\, ,
\end{equation}
where the $\mathcal{Z}_{\alpha(2^{-J}l)}$, $l=0,\ldots, 2^J-1$, are independent usual symmetric stable random measures with stability parameters $\alpha(2^{-J}l)$, $l=0,\ldots, 2^J-1$. Notice that the approximation (\ref{rem1:fieldX:eq3}) is  justified by \cite[Theorem 2.6]{falconer2012multistable}.
\end{rem}

Here are some simulations:




\begin{figure}[!h]
  \begin{center}
    \subfloat[$\al$ function]{
\includegraphics[scale=0.15]{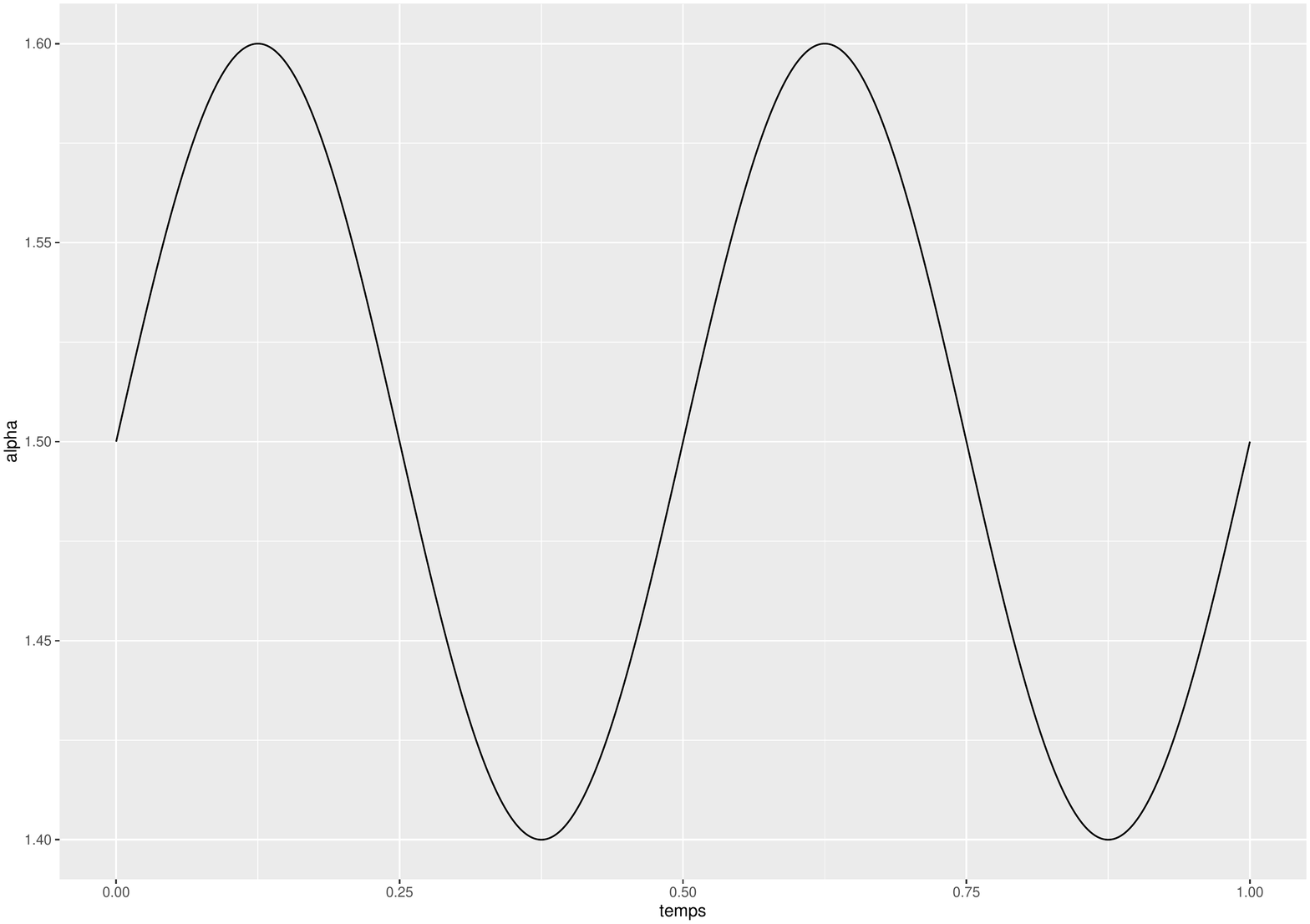}
                         }
    \subfloat[Hurst's function]{
		\includegraphics[scale=0.15]{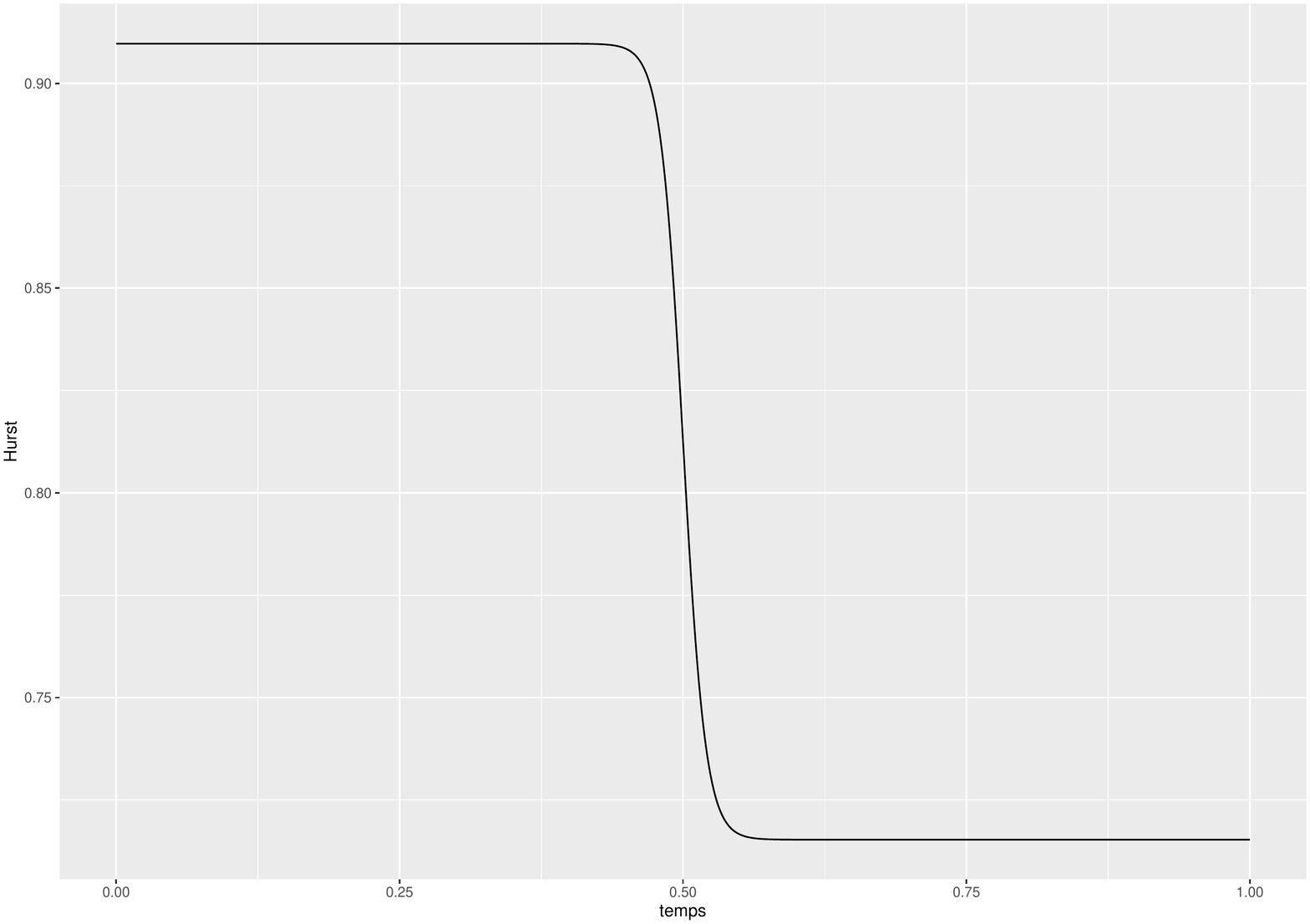}
                         }
    \subfloat[mmRL]{
      \includegraphics[scale=0.15]{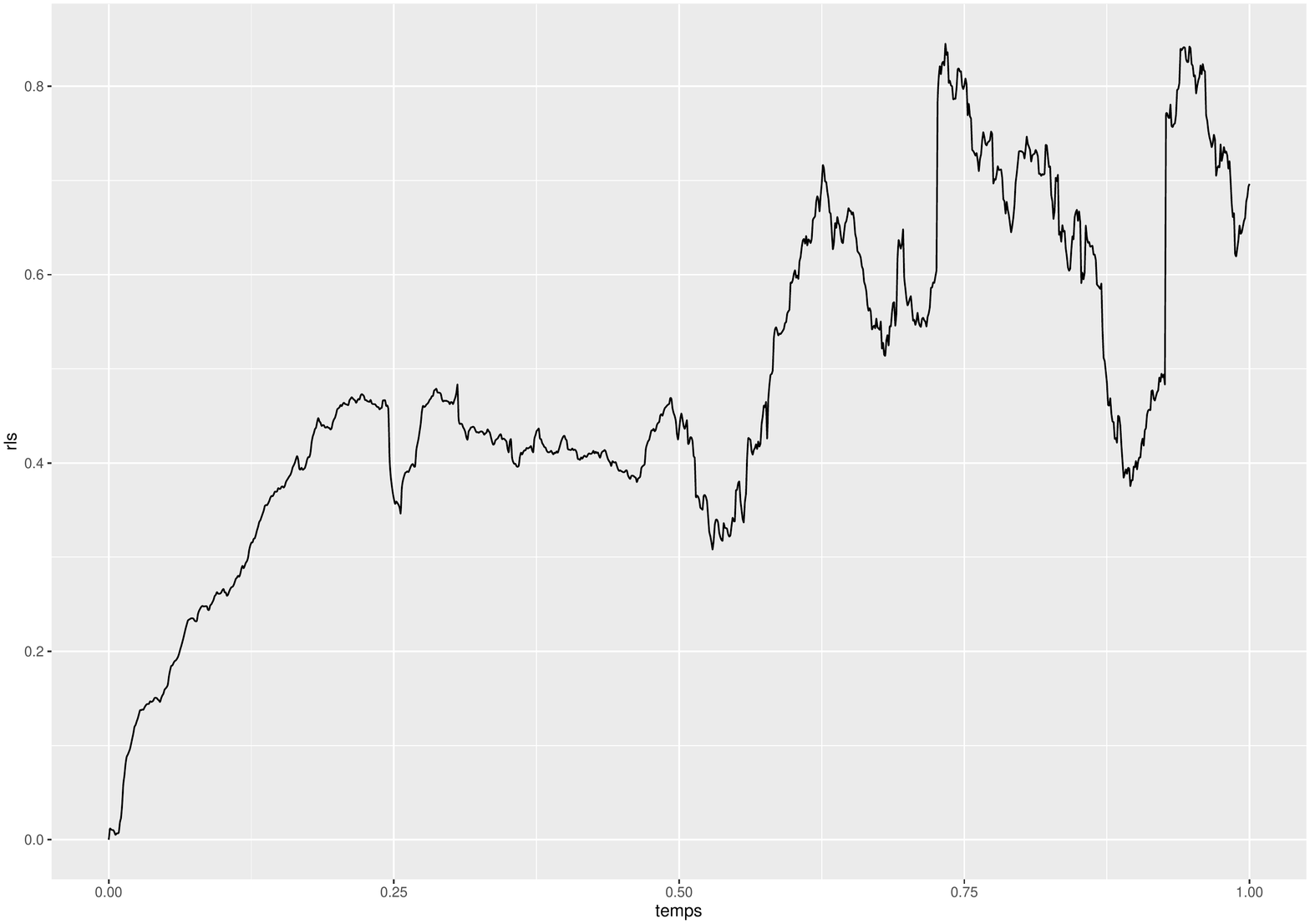}
                         }
    \caption{Multifractional Multistable Riemann-Liouville process}
    \label{fig:twosamplespath}
  \end{center}
\end{figure}

\pagebreak

For the same function $\al$ as above, 

\begin{figure}[!h]
  \begin{center}
    \subfloat[H=0.72]{
\includegraphics[scale=0.15]{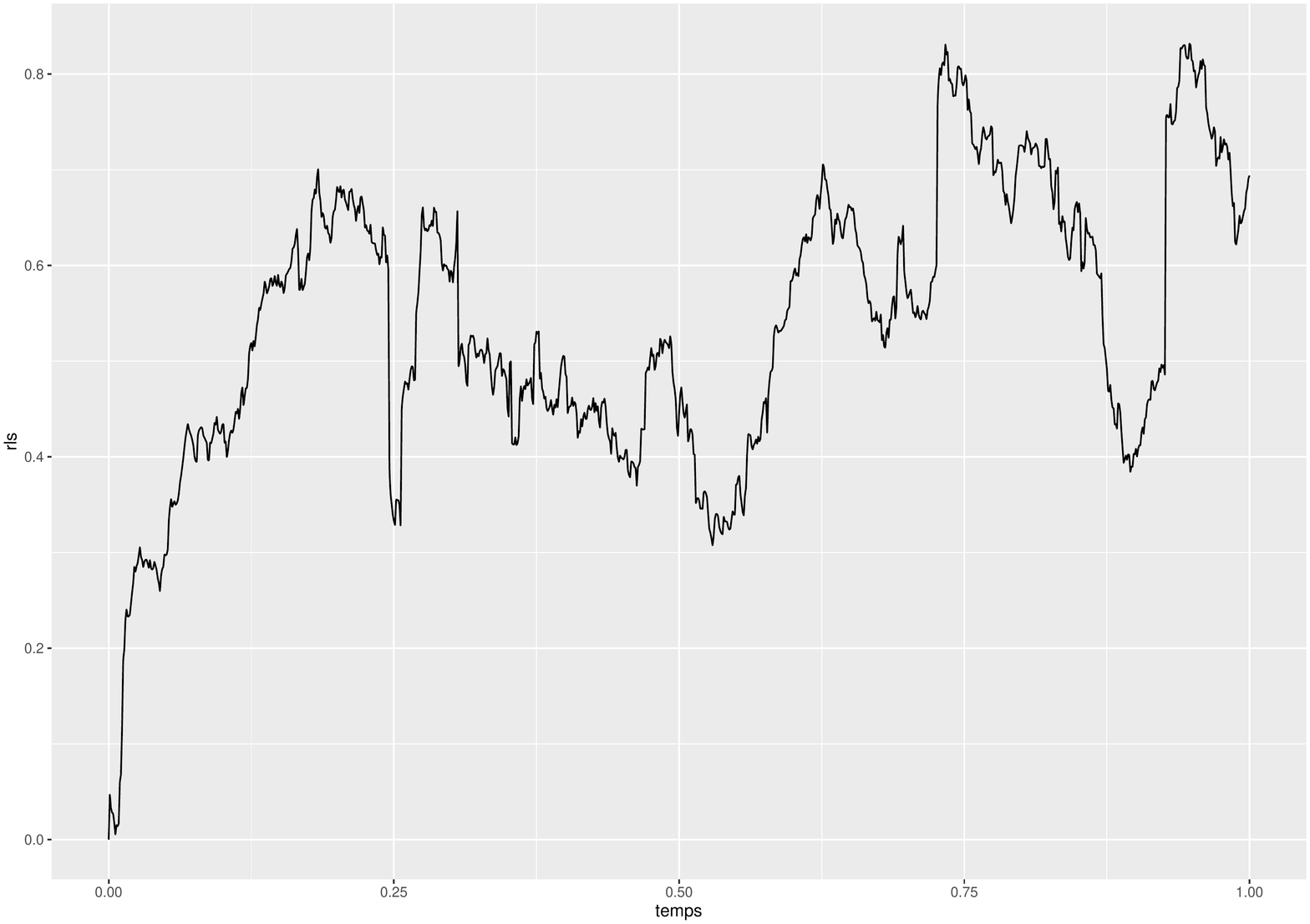}
                         }
    \subfloat[H=0.81]{
		\includegraphics[scale=0.15]{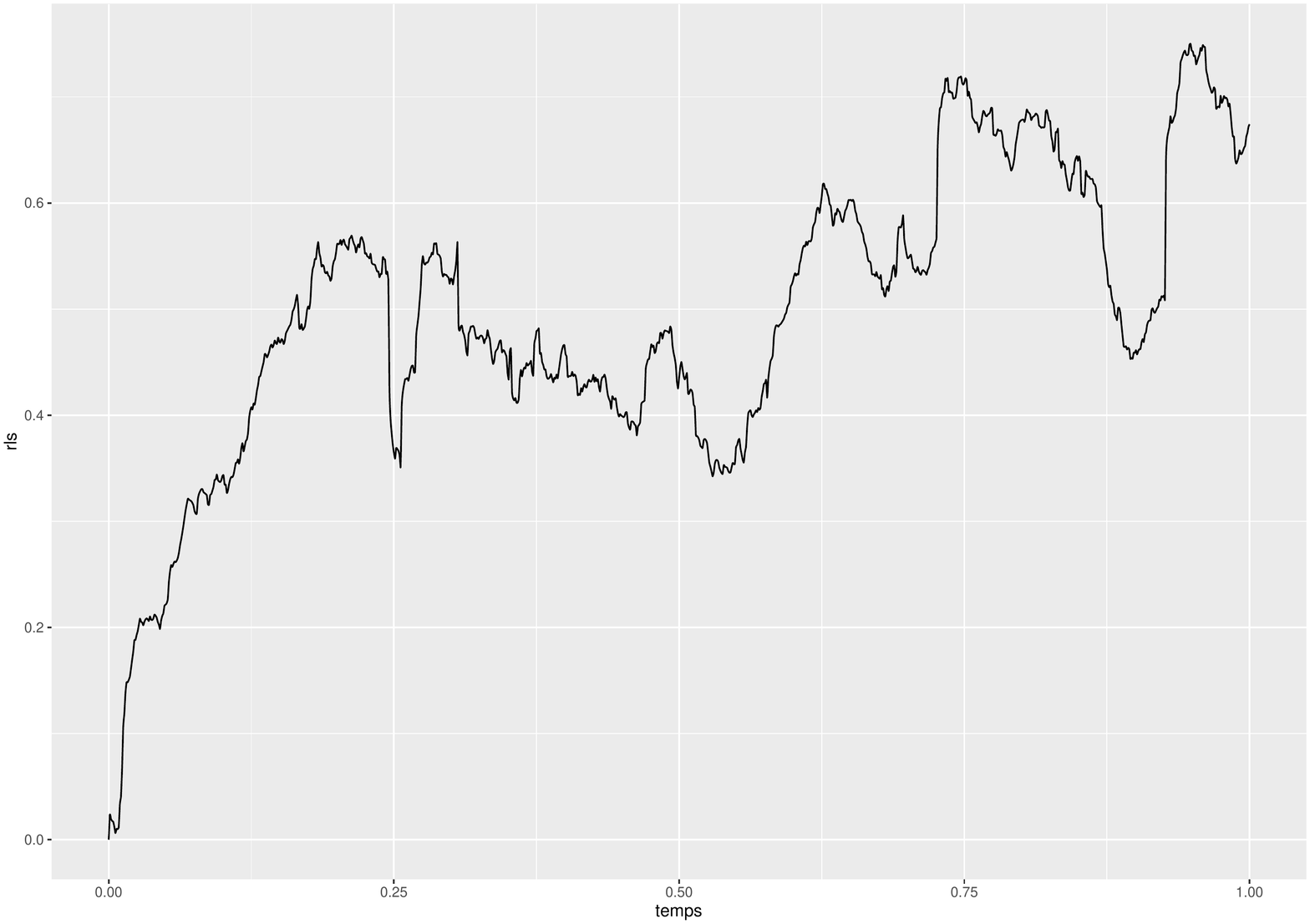}
                         }
    \subfloat[H=0.90]{
      \includegraphics[scale=0.15]{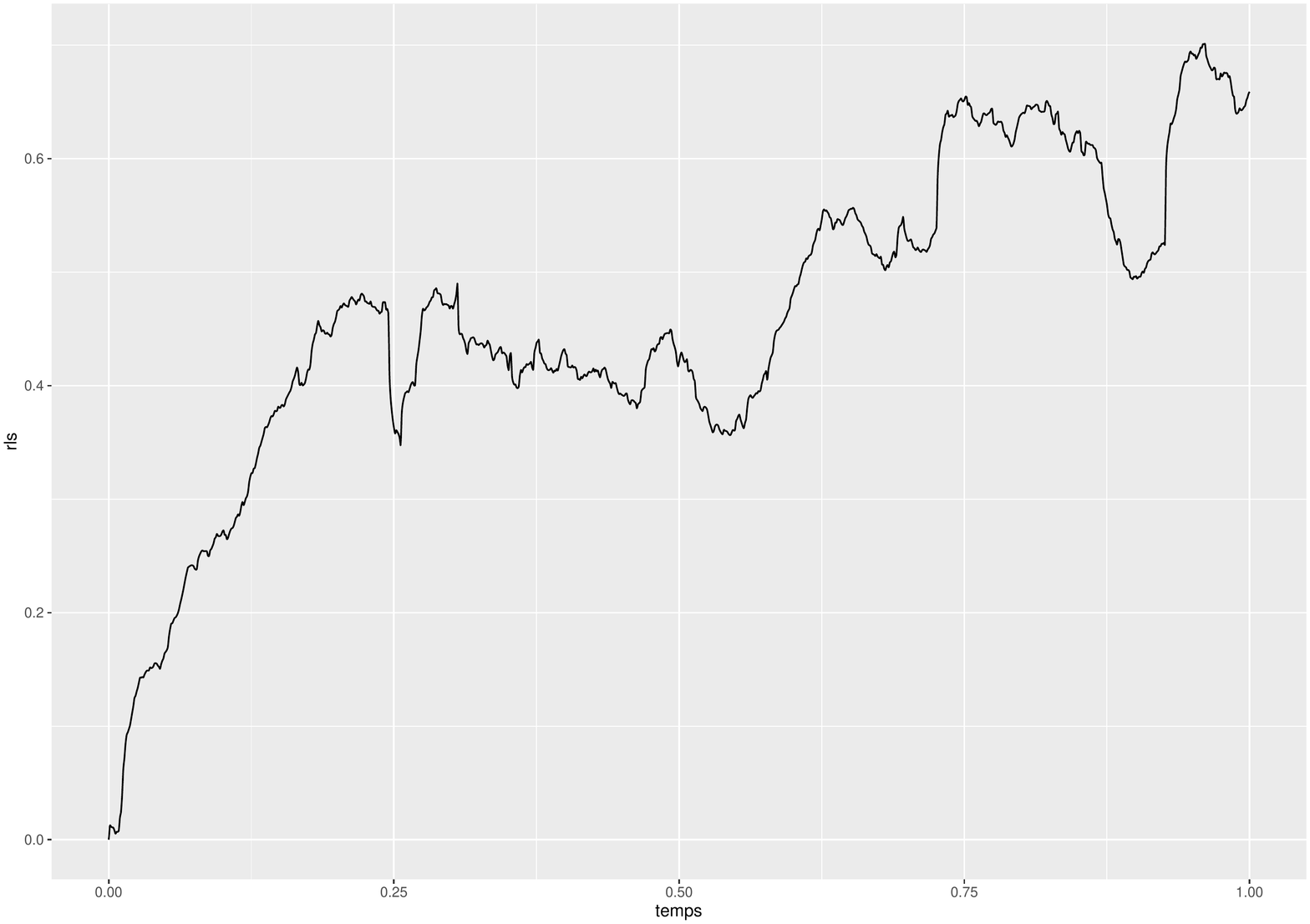}
 }
    \caption{Fractional multistable Riemann-Liouville processes}\label{fig:twosamplespath_frac2}
  \end{center}
\end{figure}

\section{Proof of the main result}\label{sec:proof}

The main two ideas of the proof of Theorem~\ref{theo:main} are: 
\begin{enumerate}
\item The use of the Abel's summation rule in order to express $\mathcal{L}_j(u,v)$ in a convenient way
(see Remark~\ref{rem:abel}).
\item The use of the Doob's maximal inequality for discrete submartingales in order to derive, for each $j\in\N$, a suitable upper for the supremum of the absolute values of the partial sums $\tau_{j,k}$, $k\in\{0,\ldots, 2^{j}-1\}$, of the multistable random variables $\epsilon_{j,0}\,,\ldots, \epsilon_{j,2^j-1}$ (see Remark~\ref{rem:mart},  Lemma~\ref{lem:bou-taujk} and its proof). 
\end{enumerate}
The first idea is borrowed from \cite{ham2015} while the second one is completely new.

\begin{rem}
\label{rem:abel}
For all $j\in\Z_+$ and $(u,v)\in [0,1]\times(1/\undal\,, 1)$, one sets
\begin{equation}
\label{eq:def:Ljuv}
\mathcal{L}_j(u,v):= \sum_{k=0}^{2^j-1} w_{j,k}(u,v) \varepsilon_{j,k}=\sum_{k=0}^{[2^ju]} w_{j,k}(u,v) \varepsilon_{j,k}\,.
\end{equation}
One mentions in passing that the last equality in (\ref{eq:def:Ljuv}), in which $[2^ju]$ denotes the integer part of
$2^j u$, follows from (\ref{A:eq:haar}) and (\ref{eq:def-wjk}). Using the Abel's summation rule one has that:
\begin{equation}
\label{eq:prop:Abel:XJuv}
\mathcal{L}_j(u,v)= \tau_{j,[2^ju]} w_{j,[2^ju]}(u,v)+\sum_{k=0}^{[2^ju]-1} \tau_{j,k} \left( w_{j,k}(u,v)-w_{j,k+1}(u,v) \right),
\end{equation}
where $\big\{\tau_{j,k}\, : \, j\in\Z_+ \mbox{ and } k\in\{0,\ldots, 2^{j}-1\}\big\}$ is the sequence of the multistable random variables defined, for all $j\in\Z_+$ and $k\in\{0,\ldots, 2^{j}-1\}$, as:
\begin{equation}
\label{def:taujk}
\tau_{j,k}:=\sum_{m=0}^k \varepsilon_{j,m}= \int_\R\Big ( \sum_{m=0}^k h(2^j s-m)\Big) \Ma{ds}.
\end{equation}
Notice that the last equality in (\ref{def:taujk}) follows from (\ref{eq:def-epjk}).
\end{rem}

\begin{rem}
\label{rem:mart}
One knows from (\ref{eq:def-epjk}) and (\ref{eq:taildist2}) that the multistable random variables $\varepsilon_{j,k}$, $j\in\N$ and $k\in\{0,\ldots, 2^j-1\}$, belong to $L^{\ga}(\O, \A, \PR)$, for all $\ga\in (0,\undal)$; which in particular means that they are in  $L^{1}(\O, \A, \PR)$ since $\undal >1$; notice that they are centered since their distributions are symmetric. Moreover, for each fixed $j~\in~\N$, the random variables $\varepsilon_{j,k}$, $k\in\{0,\ldots, 2^j-1\}$, are independent since the random measure $\mathcal{M}_\al$ is independently scattered and the supports of the functions $h(2^j \bullet-k)$, $k\in\{0,\ldots, 2^j-1\}$, are pairwise disjoints (up to Lebesgue negligible sets). Thus, in view of the first equality in (\ref{def:taujk}), it turns out that, for each fixed $j\in\N$, the sequence of random variables $\{\tau_{j,k}\}_{0\le k <2^j}$ is a discrete martingale with respect to the filtration $(\A_{j,k})_{0\le k <2^j-1}$ such that, for any $k\in\{0,\ldots, 2^j-2\}$, $\A_{j,k}$ denotes the smallest $\sigma$-algebra for which the random variables $\vep_{j,0},\ldots, \vep_{j,k}$ are measurable. 
\end{rem}

\begin{lemma}
\label{lem:bou-taujk}
There exists an event $\O^*$ of probability~$1$ such that on $\O^*$, one has, for all fixed $\zeta>1/\undal$\,, 
\begin{equation}
\label{lem:bou-taujk:eq1}
\sup_{j\in\Z_+}\Big\{(1+j)^{-\zeta}\, \sup_{0\le k < 2^j} |\tau_{j,k}|\Big\}<+\infty.
\end{equation} 
\end{lemma}

\begin{proof}[Proof of Lemma~\ref{lem:bou-taujk}] Let $\zeta>1/\undal$ and $\gamma\in [1, \undal)$ be fixed and such that 
\begin{equation}
\label{lem:bou-taujk:eq2}
\zeta>1/\gamma>1/\undal\,.
\end{equation}
Observe $z\mapsto |z|^\gamma$ is a convex function from $\R$ to $\R_+$, and one has $\ESP \big (|\tau_{j,k}|^\gamma\big)<+\infty$, for all $j$ and $k$ (see the last equality in (\ref{def:taujk}), (\ref{eq:taildist2}) and (\ref{A:eq:haar})). Thus, it follows Remark~\ref{rem:mart} and from \cite[Theorem~10.3.3 on page 354]{Dud03}, that, for each fixed $j\in\N$, the sequence of random variables $\big\{|\tau_{j,k}|^{\gamma}\big\}_{0\le k <2^j}$ is a discrete submartingale with respect to the filtration $(\A_{j,k})_{0\le k <2^j-1}$. Hence, using the Doob's maximal inequality (see \cite[Theorem 10.4.2 on page 360]{Dud03})  one has, for all positive real number $M$, 
\begin{equation}
\label{lem:bou-taujk:eq3}
\PR\Big (\sup_{0\le k < 2^j} |\tau_{j,k}|^\gamma >M\Big)\le M^{-1}\ESP\big (|\tau_{j,2^j-1}|^\gamma\big)\,.
\end{equation}
Observe that it follows from the last equality in (\ref{def:taujk}), (\ref{eq:taildist2}) and the fact that 
\[
\int_{\R} \Big |\sum_{m=0}^{2^j-1} h(2^j s-m)\Big |^{\al (s)} ds=1\,,
\] 
that 
\begin{equation}
\label{lem:bou-taujk:eq4}
\ESP\big (|\tau_{j,2^j-1}|^\gamma\big)\le\kappa_3 (\gamma),\quad\mbox{for all $j\in\Z_+$,}
\end{equation}
where $\kappa_3 (\gamma)$ is the same finite constant as in (\ref{eq:taildist2}). Next, taking in 
(\ref{lem:bou-taujk:eq3}) $M=(1+j)^{\gamma \zeta}$ and using (\ref{lem:bou-taujk:eq4}) and (\ref{lem:bou-taujk:eq2}), one obtains that
\[
\sum_{j=1}^{+\infty} \PR\Big (\sup_{0\le k < 2^j} |\tau_{j,k}| >(1+j)^{\zeta}\Big)\le \kappa_3 (\gamma) 
\sum_{j=1}^{+\infty} (1+j)^{-\gamma \zeta}<+\infty.
\]
Thus, it follows from the Borel-Cantelli's Lemma that the probability of the event 
\[
\O_\zeta^*:=\bigg\{\o\in\O\,:\,\sup_{j\in\Z_+}\Big\{(1+j)^{-\zeta}\, \sup_{0\le k < 2^j} |\tau_{j,k}(\o)|\Big\}<+\infty\bigg\}
\]
is equal to $1$. For finishing the proof, one sets 
\[
\O^*:=\bigcap_{\zeta\in \Q\cap (1/\undal\,,+\infty)} \O_\zeta^*\,
\]
where $\Q$ denotes the countable set of the rational numbers.
\end{proof}

In order to derive Theorem~\ref{theo:main}, one also needs the following five lemmas whose proofs are given in the Appendix \ref{app}. From now on, for the sake of simplicity one denotes by $I$ the interval $[0,1]$.




\begin{lemma}\label{lem:useful:2}
There exists a positive and finite constant $c_1$ such that, for any $j\in\N$, for each $(u,v)\in I\times [a,b]$ satisfying $u\ge 4.2^{-j-1}$, and for all $s\in [0,u-4.2^{-j-1}]$, the following inequality holds
\begin{align}
\label{lem:useful:2:eq13}
& \left| K_{u,v}(s)-K_{u,v}(s+2^{-j-1})-K_{u,v}(s+2.2^{-j-1})+K_{u,v}(s+3.2^{-j-1}) \right| \nonumber \\
& \le c_1 2^{-j} \big ( 2^{-j\rho_{\al}} + 2^{-j} |u-s-3.2^{-j-1}|^{a-1/\undal-2}\big)\,.
\end{align}
\end{lemma}

\begin{lemma}\label{prop1:maj:wjk}
There exists a positive and finite constant $c_2$ such that, for any $j\in\Z_+$ and $(u,v) \in I\times [a,b]$, the following inequality, in which $[2^ju]$ denotes the integer part of $2^ju$, is satisfied
\begin{equation}\label{result:prop1:maj:wjk}
\left| w_{j,[2^ju]}(u,v) \right| \le c_2 2^{-j\left(a-\frac{1}{\undal}\right)}.
\end{equation}
\end{lemma}

\begin{lemma}\label{lem:maj:I1}
There exists a positive and finite constant $c_3$ such that, for any $j\in\Z_+$ and $(u,v)\in I\times [a,b]$, one has
\begin{align}\label{result:lem:maj:I1}
I_j^1(u,v)& :=2^{j} \int_{u-2\cdot 2^{-(j+1)}}^{u-2^{-(j+1)}} \bigg |(u-s)^{v-\frac{1}{\al(s)}} - (u-s-2^{-j-1)})^{v-\frac{1}{\al(s+2^{-j-1})}} \bigg | ds \le c_3 2^{-j\left(a-\frac{1}{\undal}\right)},
\end{align}
with the convention that $I_j^1(u,v):=0$ when $u\le 2^{-(j+1)}$.
\end{lemma}

\begin{lemma}\label{lem:maj:I2}
There exists a positive and finite constant $c_4$ such that, for any $j\in\Z_+$ and $(u,v)\in I\times [a,b]$, one has
\begin{align}\label{result:lem:maj:I2}
I_j^2(u,v)& :=2^{j} \int_{u-3 \cdot 2^{-(j+1)}}^{u-2\cdot 2^{-(j+1)}} \bigg |(u-s)^{v-\frac{1}{\al(s)}} - (u-s-2^{-(j+1)})^{v-\frac{1}{\al(s+2^{-j-1})}}-(u-s-2\cdot 2^{-(j+1)})^{v-\frac{1}{\al(s+2\cdot 2^{-j-1})}} \bigg | ds \nonumber \\
&\le c_4 2^{-j\left(a-\frac{1}{\undal}\right)},
\end{align}
with the convention that $I_j^2(u,v):=0$ when $u\le 2\cdot 2^{-(j+1)}$.
\end{lemma}

\begin{lemma}\label{lem:maj:I3}
There exists a positive and finite constant $c_5$ such that, for any $j\in\Z_+$ and $(u,v)\in I\times [a,b]$, one has
\begin{align}\label{result:lem:maj:I3}
I_j^3(u,v)& :=2^{j} \int_{u-4 \cdot 2^{-(j+1)}}^{u-3\cdot 2^{-(j+1)}} \bigg |(u-s)^{v-\frac{1}{\al(s)}} - (u-s-2^{-(j+1)})^{v-\frac{1}{\al(s+2^{-j-1})}}\nonumber \\
& \qquad\qquad \qquad\qquad\qquad- (u-s-2\cdot 2^{-(j+1)})^{v-\frac{1}{\al(s+2\cdot 2^{-j-1})}} + (u-s-3\cdot 2^{-(j+1)})^{v-\frac{1}{\al(s+3\cdot 2^{-j-1})}} \bigg | ds \nonumber \\
&\le c_5 2^{-j\left(a-\frac{1}{\undal}\right)},
\end{align}
with the convention that $I_j^3(u,v):=0$ when $u\le 3\cdot 2^{-(j+1)}$.
\end{lemma}

We are now in position to prove Theorem~\ref{theo:main}.

\begin{proof}[Proof of Theorem~\ref{theo:main}]
Let $J\in\N$, $Q\in\N$,  $(u,v)\in I \times [a,b]$ and $\o \in \Omega^*$ be arbitrary and fixed. Using (\ref{theo:main:eq2}), (\ref{eq:def:Ljuv}), (\ref{eq:prop:Abel:XJuv}), the triangular inequality, (\ref{eq:def-wjk}) and (\ref{A:eq:haar}), one gets that
\begin{align}
\label{theo:main:eq5}
& \big |X^{J+Q}(u,v,\o)-X^{J}(u,v,\o)\big| = \Bigg| \sum_{j=J}^{J+Q-1} \mathcal{L}_j(u,v,\o)\Bigg| \nonumber \\
& \le \sum_{j=J}^{J+Q-1} \Big( \sup_{0\le k <2^j} |\tau_{j,k}(\o)| \Big) \Big( |w_{j,[2^ju]}(u,v)| + \sum_{k=0}^{[2^ju]-1} |w_{j,k}(u,v)-w_{j,k+1}(u,v)| \Big) \nonumber \\
& = \sum_{j=J}^{J+Q-1}  \Big( \sup_{0\le k <2^j} |\tau_{j,k}(\o)| \Big) \Bigg( |w_{j,[2^ju]}(u,v)|  \nonumber \\
& \qquad \qquad \qquad \quad + \sum_{k=0}^{[2^ju]-1} \bigg | 2^j \int_{k2^{-j}}^{(k+1/2)2^{-j}} \Big(K_{u,v}(s)-K_{u,v}(s+2^{-j-1})-K_{u,v}(s+2\cdot 2^{-j-1})+K_{u,v}(s+3\cdot2^{-j-1})\Big) ds \bigg|\Bigg) \nonumber \\
& \le  \sum_{j=J}^{J+Q-1} \Big( \sup_{0\le k <2^j} |\tau_{j,k}(\o)| \Big) \bigg ( |w_{j,[2^ju]}(u,v)|  \nonumber \\
& \qquad \qquad \qquad \qquad \qquad +  2^j \int_{0}^{u-2^{-j-1}} \Big|K_{u,v}(s)-K_{u,v}(s+2^{-j-1})-K_{u,v}(s+2\cdot 2^{-j-1})+K_{u,v}(s+3\cdot2^{-j-1})\Big| ds \bigg).
\end{align}
Next, putting together (\ref{theo:main:eq5}), (\ref{def_Kuv}) and Lemmas~\ref{lem:bou-taujk} to \ref{lem:maj:I3}, one obtains, for any fixed $\zeta>1/\undal$, that:
\begin{align}
\label{theo:main:eq6}
& \big |X^{J+Q}(u,v,\o)-X^{J}(u,v,\o)\big| \le C'(\o)  \sum_{j=J}^{J+Q-1} (1+j)^{\zeta} \Big( 2^{-j\min \{ a-\frac{1}{\undal},\rho_{\al}\}} +2^{-j}\int_{0}^{u-4\cdot 2^{-j-1}} (u-s-3\cdot 2^{-j-1})^{a-1/\undal-2}ds \Big) \nonumber \\
& \le C"(\o) \sum_{j=J}^{J+Q-1}  (1+j)^{\zeta} 2^{-j\min \{ a-\frac{1}{\undal},\rho_{\al}\}} \le  C"(\o) \sum_{j=J}^{+\infty} (1+j)^{\zeta} 2^{-j\min \{ a-\frac{1}{\undal},\rho_{\al}\}},
\end{align}
where $C'$ and $C"$ are two positive and finite random variables not depending on $J$, $Q$ and $(u,v)$. Thus, one can derive from (\ref{theo:main:eq6}) that $(X^{J}(\cdot,\cdot,\o))_{J\in\N}$ is a Cauchy sequence in the Banach space $\ce(I\times [a,b])$; its limit in this space is denoted by $\widetilde{X}(\cdot,\cdot,\o)$.

Let us now prove that (\ref{theo:main:eq4}) is satisfied. When $Q$ goes to $+\infty$, it follows from (\ref{theo:main:eq6}) that
\begin{align}
\label{theo:main:eq7}
& \big |\widetilde{X}(u,v,\o)-X^{J}(u,v,\o)\big|\le C"(\o) \sum_{j=J}^{+\infty} (1+j)^{\zeta} 2^{-j\min \{ a-\frac{1}{\undal},\rho_{\al}\}} \nonumber\\
&\le  C"(\o)(1+J)^{\zeta} 2^{-J\min \{ a-\frac{1}{\undal},\rho_{\al}\}}  \sum_{j=0}^{+\infty} (1+j)^{\zeta} 2^{-j\min \{ a-\frac{1}{\undal},\rho_{\al}\}}  \le C'''(\o) \, J^{\zeta} 2^{-J\min \{ a-\frac{1}{\undal},\rho_{\al}\}},
\end{align}
where $C'''$ is a positive and finite random variable not depending on $J$ and $(u,v)$. Thus, (\ref{theo:main:eq7}) implies that (\ref{theo:main:eq4}) holds.
\end{proof}

\appendix
\section{Appendix}\label{app}

\begin{proof}[Proof of Lemma~\ref{lem:useful:2}]
One assumes that $j\in\N$ and $(u,v)\in I\times [a,b]$ are arbitrary and such that $u\ge 4.2^{-j-1}$. Then, one denotes by $L_{u,v}$ the infinitely differentiable function on the open set $(-\infty, u)\times (1/v,+\infty)\subset\R^2$ defined as:
\begin{equation}
\label{lem:useful:2:eq0}
L_{u,v}(x,y):=(u-x)^{v-1/y}\,,\quad\mbox{for all $(x,y)\in (-\infty, u)\times (1/v,+\infty)$.}
\end{equation}
Thus, using (\ref{def_Kuv}), one has
\begin{equation}
\label{lem:useful:2:eq1}
K_{u,v}(z)=L_{u,v}\big(z,\al(z)\big)\, \quad\mbox{for all $z\in [0, u-2^{-j-1}]$.}
\end{equation}
One can derive from (\ref{lem:useful:2:eq1}) and the triangular inequality that, for every $s\in [0,u-4.2^{-j-1}]$,
\begin{align} \label{maj:ABjuv}
& \left| K_{u,v}(s)-K_{u,v}(s+2^{-j-1})-K_{u,v}(s+2.2^{-j-1})+K_{u,v}(s+3.2^{-j-1}) \right| \nonumber \\
& = \Big | L_{u,v}(s,\al(s))-L_{u,v}(s+2^{-j-1},\al(s+2^{-j-1})) \nonumber \\
& \qquad \qquad \qquad -L_{u,v}(s+2.2^{-j-1},\al(s+2.2^{-j-1}))+ L_{u,v}(s+3.2^{-j-1},\al(s+3.2^{-j-1})) \Big | \nonumber \\
& \le A_{u,v}^j(s)+B_{u,v}^j(s),
\end{align}
where
\begin{align}\label{dej:ajuv}
 A_{u,v}^j(s)& := \Big | L_{u,v}(s,\al(s))-L_{u,v}(s+2^{-j-1},\al(s+2^{-j-1}))\nonumber \\
& \qquad \qquad - L_{u,v}(s+2.2^{-j-1},\al(s))+ L_{u,v}(s+3.2^{-j-1},\al(s+2^{-j-1})) \Big |
\end{align}
and
\begin{align}\label{def:bjuv}
B_{u,v}^j(s) & := \Big | L_{u,v}(s+2.2^{-j-1},\al(s))-L_{u,v}(s+3.2^{-j-1},\al(s+2^{-j-1})) \nonumber \\
& \qquad \qquad - L_{u,v}(s+2.2^{-j-1},\al(s+2.2^{-j-1}))+ L_{u,v}(s+3.2^{-j-1},\al(s+3.2^{-j-1})) \Big |.
\end{align}

\textbf{First step:} The goal of this step is to provide a suitable upper bound for the quantity $A_{u,v}^j(s)$. \\
For any fixed $s\in [0,u-4.2^{-j-1}]$, one denotes by $g_{1,s}$ the infinitely differentiable function defined as: 
\begin{equation}
\label{lem:useful:2:eq2}
    g_{1,s}\colon\begin{cases}
    [0,2^{-j}]\longrightarrow\mathbb{R}\\
    x\longmapsto L_{u,v}(s+x,\al(s))-L_{u,v}(s+2^{-j-1}+x,\al(s+2^{-j-1})).
    \end{cases}
\end{equation}
Thus, it follows from (\ref{dej:ajuv}) that $A_{u,v}^j(s) = \big |g_{1,s}(2^{-j})-g_{1,s}(0)\big|$. Then using the mean value theorem, one obtains that $A_{u,v}^j(s) =2^{-j} |g_{1,s}'(x_*)|$, for some $x_*\in (0,2^{-j-1})$. Therefore, one can derive from (\ref{lem:useful:2:eq2}), (\ref{lem:useful:2:eq0}), the triangular inequality and the inequalities 
$|v-1/\al(s+2^{-j-1}) |<1$ and $ |v-1/\al(s) |<1$ that 
\begin{align}
\label{lem:useful:2:eq3}
& A_{u,v}^j(s) \le  2^{-j} \bigg( \Big |(u-s-2^{-j-1}-x_*)^{v-1-1/\al(s+2^{-j-1})}-(u-s-2^{-j-1}-x_*)^{v-1-1/\al(s)}\Big | \\
& \quad  +\Big |\frac{1}{\al (s+2^{-j-1})}-\frac{1}{\al (s)}\Big |(u-s-2^{-j-1}-x_*)^{v-1-1/\al(s)}+ \Big |(u-s-2^{-j-1}-x_*)^{v-1-1/\al(s)}-(u-s-x_*)^{v-1-1/\al(s)}\Big |\bigg)\, .\nonumber
\end{align}
Next, notice that it follows from the assumption $\al\in\ce^{1+\rho_\al}(I)$, (\ref{eq:condal}), $x_*\in (0,2^{-j-1})$ and $v\in [a,b]$, that 
\begin{eqnarray}
\label{lem:useful:2:eq4}
\Big |\frac{1}{\al (s+2^{-j-1})}-\frac{1}{\al (s)}\Big |(u-s-2^{-j-1}-x_*)^{v-1-1/\al(s)} & \le & c_1 2^{-j-1}(u-s-2^{-j-1}-x_*)^{v-1-1/\al(s)} \nonumber\\
&\le & c_1 2^{-j-1}(u-s-3.2^{-j-1})^{a-1-1/\undal}\,,
\end{eqnarray}
where $c_1$ is a constant not depending on $j,u,s,v$. Thus, using (\ref{lem:useful:2:eq3}) and (\ref{lem:useful:2:eq4}), and applying the mean value theorem to the functions:

\begin{equation*}
    g_{2,s,x_*}\colon\begin{cases}
    [0,2^{-j-1}]\rightarrow\mathbb{R}\\
    w\longmapsto (u-s-x_*-w)^{v-1-\frac{1}{\al(s)}}
    \end{cases}
\end{equation*}

\begin{equation*}
    g_{3,s,x_*}\colon\begin{cases}
    [\alpha(s)\wedge \alpha(s+2^{-j-1}),\alpha(s)\vee \alpha(s+2^{-j-1})]\rightarrow\mathbb{R}\\
    z\longmapsto (u-s-2^{-j-1}-x_*)^{v-1-\frac{1}{z}},
    \end{cases}
\end{equation*}

one obtains that 
\begin{align}
\label{lem:useful:2:eq5}
& A_{u,v}^j(s) \le c_2 2^{-2j} \Big( (u-s-3.2^{-j-1})^{a-1-1/\undal} +(u-s-3.2^{-j-1})^{a-2-1/\undal} \nonumber \\
& \qquad \qquad \qquad \qquad \qquad \qquad \qquad + |\log (u-s-2^{-j-1})| (u-s-3.2^{-j-1})^{a-1-1/\undal} \Big) \nonumber \\
& \le c_2 2^{-2j} \Big(2(u-s-3.2^{-j-1})^{a-2-1/\undal} + |\log (u-s-2^{-j-1})| (u-s-3.2^{-j-1})^{a-1-1/\undal} \Big)\,, 
\end{align}
where $c_2$ is a constant not depending on $j,u,s,v$. Finally, combining (\ref{lem:useful:2:eq5}) and the inequality 
$ |\log(x)|\le x^{-1}$, for all $x\in (0,1]$, one gets that
\begin{equation}\label{maj:Ajuv}
A_{u,v}^j(s) \le c_3 2^{-2j} (u-s-3.2^{-j-1})^{a-2-1/\undal}\, ,
\end{equation}
where $c_3$ is a constant not depending on $j,u,s,v$. 

\textbf{Second step:} The goal of this step is to provide a suitable upper bound for the quantity $B_{u,v}^j(s)$. \\
In view of (\ref{def:bjuv}), the quantity  $B_{u,v}^j(s)$ can be rewritten as:
\begin{align*}
B_{u,v}^j(s)& = \Big |\big (L_{u,v}(s+2.2^{-j-1},\al(s))- L_{u,v}(s+2.2^{-j-1},\al(s+2.2^{-j-1}))\big) \nonumber \\
& \qquad \qquad -\big(L_{u,v}(s+3.2^{-j-1},\al(s+2^{-j-1}))- L_{u,v}(s+3.2^{-j-1},\al(s+3.2^{-j-1}))\big) \Big |.
\end{align*}
Thus applying the mean value theorem to the functions

\begin{equation*}
    g_{4,s}\colon\begin{cases}
    [\alpha(s)\wedge \alpha(s+2.2^{-j-1}),\alpha(s)\vee \alpha(s+2.2^{-j-1})]\longrightarrow\mathbb{R}\\
    y\longmapsto L_{u,v}(s+2.2^{-j-1},y),
    \end{cases}
\end{equation*}
 \begin{equation*}
    g_{5,s}\colon\begin{cases}
    [\alpha(s+2^{-j-1})\wedge \alpha(s+3.2^{-j-1}),\alpha(s+2^{-j-1})\vee \alpha(s+3.2^{-j-1})]\longrightarrow\mathbb{R}\\
    y\longmapsto L_{u,v}(s+3.2^{-j-1},y).
    \end{cases}
\end{equation*}
and putting together the triangular inequality, the assumption $\al\in\ce^{1+\rho_{\al}}(I)$, (\ref{eq:condal}), $v\in [a,b]$, and the equality 
\[
df-gh=d(f-h)+h(d-g)\,, \quad\mbox{for all $d,f,g,h\in\R$,}
\]
one obtains, for some
\begin{equation}
\label{lem:useful:2:eq6}
y_*\in  \left(\alpha(s)\wedge \alpha(s+2.2^{-j-1}),\alpha(s)\vee \alpha(s+2.2^{-j-1})\right)
\end{equation}
 and 
 \begin{equation}
\label{lem:useful:2:eq7}
 y_{**}\in \left(\alpha(s+2^{-j-1})\wedge \alpha(s+3.2^{-j-1}),\alpha(s+2^{-j-1})\vee \alpha(s+3.2^{-j-1})\right),
\end{equation}
that 
\begin{align}
\label{lem:useful:2:eq8}
B_{u,v}^j(s) & \le c_4  2^{-j(1+\rho_{\al})} \Big|\frac{1}{y_{*}^2}\log(u-s-2.2^{-j-1})\Big |(u-s-2.2^{-j-1})^{a-1/\undal}  \nonumber \\
& + c_4 2^{-j} \Big | \big (\frac{1}{y_*^2}-\frac{1}{y_{**}^2} \big) \log(u-s-2.2^{-j-1})\Big| (u-s-2.2^{-j-1})^{a-1/\undal} \nonumber \\
& + c_4 2^{-j} \left| \frac{1}{y_{**}^2}\log(u-s-2.2^{-j-1})\right| \left| (u-s-2.2^{-j-1})^{v-1/y_*} - (u-s-2.2^{-j-1})^{v-1/y_{**}} \right| \nonumber \\
& + c_4 2^{-j} \left| \frac{1}{y_{**}^2} \left( (u-s-2.2^{-j-1})^{v-1/y_{**}}\log(u-s-2.2^{-j-1})-(u-s-3.2^{-j-1})^{v-1/y_{**}}\log(u-s-3.2^{-j-1})
 \right) \right|\,,
\end{align}
where $c_4$ is a constant not depending on $j,u,s,v, y_*, y_{**}$\,. Next, notice that using (\ref{eq:condal}), (\ref{lem:useful:2:eq6}), (\ref{lem:useful:2:eq7}) and the assumption $\al\in\ce^{1+\rho_\al}(I)$, one gets that 
\begin{equation}
\label{lem:useful:2:eq8bis}
\max\Big\{\frac{1}{y_*^2},\frac{1}{y_{**}^2} \Big \}<1
\end{equation}
and
\begin{equation}
\label{lem:useful:2:eq9}
\Big |\frac{1}{y_*^2}-\frac{1}{y_{**}^2} \Big |=\Big |\frac{y_{**}^2-y_{*}^2}{y_{**}^2y_{*}^2}\Big |\le 4\,|y_{**}-y_*|\le 4\,\big |\max_{0\le i \le 3} \al(s+i2^{-j-1})-\min_{0\le i \le 3} \al(s+i2^{-j-1})\big|
 \le c_5 2^{-j} ,
\end{equation}
where $c_5$ is a constant not depending on $j,u,s,v, y_*, y_{**}$\,. Also, notice that applying the mean value theorem to the function 
\begin{equation*}
    g_{6,s}\colon\begin{cases}
    [y_*\wedge  y_{**},y_*\vee  y_{**}]\longrightarrow\mathbb{R}\\
    x\longmapsto  (u-s-2.2^{-j-1})^{v-1/x}
    \end{cases}
\end{equation*}
and using $v\in[a,b]$, $y_*,y_{**}\in (\undal,\oval)$, the second and the third inequality in (\ref{lem:useful:2:eq9}), one obtains that
\begin{equation}
\label{lem:useful:2:eq10}
\left| (u-s-2.2^{-j-1})^{v-1/y_*} - (u-s-2.2^{-j-1})^{v-1/y_{**}} \right|\le c_6 2^{-j}(u-s-2.2^{-j-1})^{a-1/\undal} \big |\log(u-s-2.2^{-j-1})\big|\, ,
\end{equation}
where $c_6$ is a constant not depending on $j,u,s,v, y_*, y_{**}$\,. Moreover, notice that applying the mean value theorem to the function 
\begin{equation*}
    g_{7,s,y_{**}}\colon\begin{cases}
    [0,2^{-j-1}]\longrightarrow\mathbb{R}\\
    x \longmapsto (u-s-2.2^{-j-1}-x)^{v-1/y_{**}}\log(u-s-2.2^{-j-1}-x)
    \end{cases}
\end{equation*}
and making use of $v\in[a,b]$ and $y_*,y_{**}\in (\undal,\oval)$,
it follows that 
\begin{eqnarray}
\label{lem:useful:2:eq11}
&& \left| (u-s-2.2^{-j-1})^{v-1/y_{**}}\log(u-s-2.2^{-j-1})-(u-s-3.2^{-j-1})^{v-1/y_{**}}\log(u-s-3.2^{-j-1})\right|\nonumber\\
&& \le c_7 2^{-j} (u-s-3.2^{-j-1})^{a-1/\undal-1}\Big (1+\log(u-s-3.2^{-j-1})\Big)\,,
\end{eqnarray}
where $c_7$ is a constant not depending on $j,u,s,v, y_*, y_{**}$\,.
%
Next putting together (\ref{lem:useful:2:eq8}) to (\ref{lem:useful:2:eq11}), one gets that
\begin{align}
\label{lem:useful:2:eq12}
B_{u,v}^j(s) & \le c_4  2^{-j(1+\rho_{\al})} \big|\log(u-s-2.2^{-j-1})\big |(u-s-2.2^{-j-1})^{a-1/\undal} \nonumber \\
& \qquad + c_8 2^{-2j} \bigg ( (u-s-2.2^{-j-1})^{a-1/\undal}\big |\log(u-s-2.2^{-j-1})\big | \nonumber \\
& \qquad \qquad +  (u-s-2.2^{-j-1})^{a-1/\undal} \log^2(u-s-2.2^{-j-1}) \nonumber \\
& \qquad \qquad +(u-s-3.2^{-j-1})^{a-1/\undal-1}\Big (1+\log(u-s-3.2^{-j-1})\Big) \bigg)\,, 
\end{align}
where $c_8$ is a constant not depending on $j,u,s,v, y_*, y_{**}$\,. Then, one can derive from (\ref{lem:useful:2:eq12}) and the inequalities $\sup_{z\in(0,1]} |\log(z)| z^{a-1/\undal} <+\infty$, $\sup_{z\in(0,1]} |\log^2(z)| z^{a-1/\undal} <+\infty$ and $ |\log(x)|\le |x|^{-1}$, for all $x\in (0,1]$, that
\begin{equation}\label{maj:Bjuv}
B_{u,v}^j(s) \le c_9 \big ( 2^{-j(1+\rho_{\al})} + 2^{-2j} |u-s-3.2^{-j-1}|^{a-1/\undal-2}\big)\,,
\end{equation}
where $c_9$ is a constant not depending on $j,u,s,v$\,. Finally combining (\ref{maj:ABjuv}), (\ref{maj:Ajuv}) and (\ref{maj:Bjuv}), one obtains (\ref{lem:useful:2:eq13}).
\end{proof}

\begin{proof}[Proof of Lemma~\ref{prop1:maj:wjk}] It easily follows from (\ref{eq:def-wjk}), (\ref{A:eq:haar}), (\ref{def_Kuv}), (\ref{eq:condal}) and the assumption that $v\in [a,b]$ that 
\[
\left| w_{j,[2^ju]}(u,v) \right| \le 2^{j} \int_{2^{-j}[2^ju]}^{u} (u-s)^{v-\frac{1}{\alpha(s)}} ds\le  2^{j} \int_{2^{-j}[2^ju]}^{u} 2^{-j (a-1/\undal)} ds\le 2^{-j (a-1/\undal)}. 
\]
\end{proof}

%

The proofs of Lemmas~\ref{lem:maj:I1}, \ref{lem:maj:I2} and \ref{lem:maj:I3} are very similar so we only give that of Lemma~\ref{lem:maj:I3}.

\begin{proof}[Proof of Lemma~\ref{lem:maj:I3}] 
Let $j\in\Z_+$ and $(u,v)\in I\times [a,b]$ be arbitrary and such that $u> 3\cdot 2^{-(j+1)}$. In view of the assumptions on $(u,v)$ and (\ref{eq:condal}), it can easily be seen that, for all $s\in [u-4 \cdot 2^{-(j+1)}, u-3 \cdot 2^{-(j+1)}]$
and for any $q\in \{0,1,2,3\}$, one has
\[
0\le \big (u-s-q\cdot 2^{-(j+1)}\big)^{v-\frac{1}{\al(s+q\cdot 2^{-j-1})}}\le \big (7\cdot 2^{-(j+1)}\big)^{v-\frac{1}{\al(s+q\cdot 2^{-j-1})}} \le 7^{b-1/\oval}\cdot 2^{-(j+1)(a-1/\undal)}\,.
\]
Thus, using the triangular inequality one gets that
\begin{align*}
I_j^3(u,v)& :=2^{j} \int_{u-4 \cdot 2^{-(j+1)}}^{u-3\cdot 2^{-(j+1)}} \Big |(u-s)^{v-\frac{1}{\al(s)}} - \big (u-s-2^{-(j+1)}\big)^{v-\frac{1}{\al(s+2^{-j-1})}}\nonumber \\
& \qquad\qquad \qquad\qquad\qquad- \big (u-s-2\cdot 2^{-(j+1)}\big)^{v-\frac{1}{\al(s+2\cdot 2^{-j-1})}} + \big (u-s-3\cdot 2^{-(j+1)}\big)^{v-\frac{1}{\al(s+3\cdot 2^{-j-1})}} \Big | ds \nonumber \\
&\le 2^{j} \int_{u-4 \cdot 2^{-(j+1)}}^{u-3\cdot 2^{-(j+1)}}  4\cdot7^{b-1/\oval}\cdot 2^{-(j+1)(a-1/\undal)} ds
= 2^{-1}\cdot 4\cdot7^{b-1/\oval}\cdot 2^{-(j+1)(a-1/\undal)}\,,
\end{align*}  
which shows that (\ref{result:lem:maj:I3}) is satisfied.

\end{proof}

\section*{Acknowledgements}
This work has been partially supported by the Labex CEMPI (ANR-11-LABX-0007-01) and the GDR 3475 (Analyse Multifractale).


\begin{thebibliography}{12}

\bibitem{Aya18}
Ayache, A. (2018)
\emph{Multifractional stochastic fields}, World Scientific.

\bibitem{Aya13}
Ayache, A. (2013)
\emph{Sharp estimates on the tail behavior of a multistable distribution}, Statistics and Probability Letters 
\textbf{83} (3), 680--688. 

%
%
%
%
%
%
%


%
%
%




%

\bibitem{Dud03}
Dudley, R.M. (2003)
\emph{Real analysis and probability (second edition)}, 74, Cambridge studies in advanced mathematics.


\bibitem{FL09}
Falconer, K. and L\'evy~V\'ehel, J. (2009) \emph{Multifractional multistable and other
  processes with prescribed local form.} Journal of Theoretical Probability
  \textbf{22}~(2), 375--401.


%

\bibitem{FGL09}
Le~Gu\'evel, R. and L\'evy~V\'ehel, J. (2012) \emph{A {F}erguson-{K}lass-{L}epage series
  representation of multistable multifractional motions and related processes.}
  Bernoulli \textbf{18}~(4), 1099--1127.



\bibitem{falconer2012multistable}
Falconer, K. and Liu, L. (2012)
\emph{Multistable processes and localizability}, Stochastic Models \textbf{28}~(3), 503--526.

\bibitem{haar}
Haar, A. (1910)
\emph{ Zur theorie der orthogonalen fuctionnensysteme}, 
Mathematische Annalen \textbf{69}, 331--371.

\bibitem{ham2015}
Hamonier, J. (2015)
\emph{Multifractional Stable Motion: representation via Haar basis}, Stochastic Processes and their Applications \textbf{125}~(3), 1127--1147.


%
%
%
%
%

\bibitem{samorodnitsky:taqqu:1994book}
Samorodnitsky, G. and Taqqu, M.~S. (1994)
\newblock {\em {\it {S}table {N}on-{G}aussian {P}rocesses: {S}tochastic
  {M}odels with {I}nfinite {V}ariance}}.
\newblock Chapman and Hall, New York, London.

%

\bibitem{stoev2004stochastic}
Stoev, S. and Taqqu, M. S. (2004)
\emph{Stochastic properties of the linear multifractional stable motion}, Advances in Applied Probability,
\textbf{36}, 1085--1115.

\bibitem{stoev2005path}
Stoev, S. and Taqqu, M. S. (2005)
\emph{Path properties of the linear multifractional stable motion}, Fractals
\textbf{13}, 157--178.





%

\end{thebibliography}
\end{document}